\documentclass{article}

\usepackage[english]{babel}
\usepackage[T1]{fontenc}
\usepackage[latin1]{inputenc}
\usepackage{amsmath,amssymb} 
\usepackage{amsthm}
\usepackage{amscd}                  
\usepackage[all]{xy}                
\usepackage{xspace}
\CompileMatrices

\newcommand{\eg}{e.g.,\xspace}

\newcommand{\s}{\ensuremath{\mathbf{S}}}    
\newcommand{\Z}{\ensuremath{\mathbf{Z}}}    
\newcommand{\Q}{\ensuremath{\mathbf{Q}}}    
\newcommand{\tilZ}{\ensuremath{\tilde{\Z}}} 

\newcommand{\kata}{\ensuremath{\mathcal{A}}}

\newcommand{\hp}{\ensuremath{\rightarrow}}       
\newcommand{\lhp}{\ensuremath{\longrightarrow}}  
\newcommand{\xhp}[1][{}]{\ensuremath{\xrightarrow}}
\newcommand{\xvp}[1][{}]{\ensuremath{\xleftarrow}}

\newcommand{\p}{\widehat{{}_p}}

\newcommand{\Lim}[1][{}]{\operatornamewithlimits{lim}_{\overleftarrow{#1}}}
\newcommand{\holim}[1][{}]{\operatornamewithlimits{holim}_{\overleftarrow{#1}
}}

\newcommand{\CDlim}[1][{}]{\underset{\overleftarrow{#1}}{\operatorname{lim}}}

\theoremstyle{plain}
\newtheorem{theo}[subsection]{Theorem}
\newtheorem{prop}[subsection]{Proposition}
\newtheorem{lemma}[subsection]{Lemma}
\newtheorem{cor}[subsection]{Corollary}

\theoremstyle{definition}

\newtheorem{ex}[subsection]{Example}
\newtheorem{remark}[subsection]{Remark}

\theoremstyle{remark}

\begin{document}
\title{Excision for K-theory of connective ring spectra}
\author{Bj{\o}rn Ian Dundas and Harald {\O}yen Kittang}

\maketitle

\numberwithin{equation}{section} 

\begin{abstract}
We extend Geisser and Hesselholt's result on ``bi-relative K-theory''
from discrete rings to connective ring
spectra.  That is, if $\mathcal A$ is a homotopy cartesian $n$-cube of
ring spectra (satisfying connectivity hypotheses), then the $(n+1)$-cube
induced by the
cyclotomic trace $$K(\mathcal A)\to TC(\mathcal A)$$
is homotopy cartesian after  profinite completion.  In other words, the fiber
of the profinitely completed cyclotomic trace satisfies excision.
\end{abstract}

\section{Introduction}
\label{sec:intro}
Topological K-theory is a cohomology theory.  Most importantly it satisfies
excision, so if for instance
$X$ is a CW-complex defined by gluing two subcomplexes $X^1$ and $X^2$ along
their intersection $X^{12}$, then the Mayer-Vietoris sequence
$$\dots\to K^0(X)\to K^0(X^1)\oplus K^0(X^2)\to K^0(X^{12})\to K^1(X)\to\dots$$
is exact.  In other words, the square of spectra
$$
\begin{CD}
  K(X)@>>>K(X^1)\\@VVV@VVV\\K(X^2)@>>>K(X^{12})
\end{CD}
$$
is homotopy cartesian.

This is not true in algebraic K-theory: given maps $f^2\colon A^2\to
A^{12}$ and $f^1\colon A^1\to A^{12}$ of rings, let
$A^0=A^1\times_{A^{12}}A^2$ be the pull back (corresponding in the
commutative case to $\operatorname{Spec}(A^0)$ being formed by gluing
$\operatorname{Spec}(A^1)$ and $\operatorname{Spec}(A^2)$ along
$\operatorname{Spec}(A^{12})$). Then
$$
\begin{CD}
  K(A^0)@>>>K(A^1)\\@VVV@VVK(f^1)V\\K(A^2)@>K(f^2)>>K(A^{12})
\end{CD}
$$
need not be homotopy cartesian. It is true that under surjectivity
conditions on $f^1$ and $f^2$ the Mayer-Vietoris sequence is exact in
low degrees, but this does not continue in higher degrees. See
\cite{swan} for an amusing account, for instance showing that even if
all maps in the square are surjective, there does not exist a functor
$K_3$ such that the Mayer-Vietoris sequence can be extended.

In a series of papers (\cite{geller83}, \cite{geller86},
\cite{geller89kabi}, \cite{geller89:_Kcurves}) Geller, Reid and Weibel
explored the idea that cyclic homology should be a precise measure for
the failure of excision in the algebraic K-theory of $\Q$-algebras,
and did some conjectural calculations. The problem remained open
(although it IS an exercise in \cite{Loday}), until Corti\~nas
released a preprint \cite{cortinas06:_k} claiming the conjecture using
Suslin and Wodzicki's results on nonunital rings \cite{suswod2}.

In a recent preprint \cite{geisser06:_bi_k} Geisser and Hesselholt give the
corresponding result after profinite completion, with the difference that
cyclic homology has to be replaced by topological cyclic homology $TC$.  The
result from \cite{geisser06:_bi_k} we generalize is the following.  Let
\begin{equation*}
   \mathcal A=\left\{\begin{CD}
    A^0 @>>> A^1 \\
    @VVV  @VV{f^1}V\\
    A^2 @>>{f^2}> A^{12}
  \end{CD}\right\}
\end{equation*}
be a cartesian square of discrete rings with $f^1$ surjective, then
the cube $K(\mathcal A)\to TC(\mathcal A)$ is homotopy cartesian after
profinite completion. A word of explanation: $K(\mathcal A)$ is the
square of spectra
\begin{equation*}
   \begin{CD}
    K(A^0) @>>> K(A^1) \\
    @VVV  @VV{K(f^1)}V\\
    K(A^2) @>{K(f^2)}>> K(A^{12})
  \end{CD}
\end{equation*}
and similarly for $TC(\mathcal A)$. The cyclotomic trace $K\to TC$
then gives a map of squares.  Considering the map of squares as a
cube, the theorem states that this cube is homotopy cartesian after
profinite completion.

Another appealing formulation is that the homotopy fiber of the profinitely
completed cyclotomic trace satisfies excision.

In this paper we extend the theorem from rings to ring spectra: let $\s$ be
the sphere spectrum in any of the popular theories of spectra with strict
symmetric monoidal smash product; then we have the following result.

\begin{theo}
  \label{thm:MainIntro}
  Consider a homotopy cartesian diagram
  \kata\ of connective \s-algebras
\begin{equation*}
   \begin{CD}
    A^0 @>>> A^1 \\
    @VVV  @VV{f^1}V\\
    A^2 @>>{f^2}> A^{12}
  \end{CD}
\end{equation*}
where $f^1$ is $0$-connected. Then the
the cube
\[K(\kata)\lhp TC(\kata)\]
induced by the cyclotomic trace, is homotopy cartesian after profinite
completion.
\end{theo}

The proof of the theorem is delightfully simple. It follows by the
theorems of McCarthy \cite{mccarthy97}, the first author
\cite{dundas97:_relat_k}, and Geisser and Hesselholt
\cite{geisser06:_bi_k} in addition to an elementary observation about
homotopy cartesian diagrams of ring spectra.

Actually we prove stronger and more technical result in proposition
\ref{theo:strong} and then show that the conditions in theorem
\ref{thm:MainIntro} imply those of proposition \ref{theo:strong}.

We mention that theorem \ref{thm:MainIntro} holds integrally by work
in progress of the second author.

\begin{ex}
  As an example, let $k$ be a connective $\s$-algebra, and consider the
``cusp'' $C$ over $k$ gotten by the homotopy pullback
 $$ \begin{CD}
    C@>>>k[t]\\@VVV@VVV\\k@>>>k[t]/t^2
  \end{CD}
$$
(if $k$ is a discrete ring, $C\cong k[x,y]/(x^2-y^3)$, hence the name).
Letting $F$ be the profinite completion of the homotopy fiber of the
cyclotomic trace, the diagram remains homotopy cartesian after applying $F$.
However, by \cite{dundas97:_relat_k} the map $F(k)\to F(k[t]/t^2)$ is an
equivalence, and so $F(C)\to F(k[t])$ is an equivalence.
The rightmost term
may then be calculated from $Nil$-terms (if $k$ is not ``regular'' the
$Nil$-term in $K(k[t])\simeq K(k)\times Nil^K_k$ will not vanish) and
$TC(k[t])$.

Hence, one can calculate $K(C)$ if one can describe
$TC(C)$, $TC(k[t])$, $K(k)$ and the $Nil$-terms (and
all the maps connecting them).
\end{ex}

\begin{remark}
One might be tempted to believe that the crucial condition on our square of
$\s$-algebras is that it is homotopy cartesian, but unfortunately the
conclusion of the theorem is false without the surjectivity assumption on
$\pi_0f^1$.  A counterexample can be derived without calculations from the
homotopy cartesian square (the maps are the natural inclusions)
$$
\begin{CD}
  \mathbf \Z@>>>\Z[t]\\@VVV@VVV\\\Z[t^{-1}]@>>>\Z[t,t^{-1}]
\end{CD}
$$
and its sibling with the $p$-adic integers $\Z\p$ instead of $\Z$. By
the fundamental theorem of algebraic K-theory, the iterated fibers of
the two K-theory squares are $K(\Z)$ and $K(\Z\p)$ respectively. They
are very different: $K_1(\Z)\cong\Z/2$ and
$K_1(\Z\p)\cong\Z/(p-1)\times\Z\p$.  On the other hand, the
topological cyclic homology of the integral and $p$-adic square agree
after $p$-completion.

The example above has the deficiency that Milnor's theorem
\cite[section IX.5]{Bass68} does not apply: the associated square of
categories of finitely generated projective modules is not a fiber
square.  We know of no examples of squares of rings for which the
Milnor theorem applies where the conclusion of the main theorem does
not hold.

A natural conjecture would be that the fiber of the cyclotomic trace
takes fiber squares of exact categories to homotopy cartesian squares.
Beyond the obvious extensions that follow from the theorems of
Corti\~nas and Geisser-Hesselholt, the case of the projective line is
our only support for this conjecture.

There is a direct proof of the extension of Geisser and Hesselholt's
theorem to simplicial rings not using Goodwillie's conjecture
\cite{mccarthy97}, \cite{dundas97:_relat_k}.  This proof is
interesting in that it gives a hands on approach to the problem, and
conceivably a way to weaken the conditions of the theorem.  We will
not pursue those questions here.
\end{remark}

\subsection{Plan}In section \ref{sec:suff} we prove a proposition that turns
out to be stronger than the main theorem \ref{thm:MainIntro}.  We do not
require the square of $\s$-algebras to be homotopy cartesian, but rather
impose criteria on the path components.

In section \ref{sec:reduction} we
address the problem that $\pi_0$ does not send homotopy cartesian squares to
cartesian squares.
We also prove some multirelative extensions.

\subsection{Conventions}
The algebraic K-theory discussed in this paper is the nonconnective
version of algebraic K-theory as defined by Thomason \cite[section 6]{TT}
extended to connective $\s$-algebras.  Thomason's construction is
functorial, and is also performed on the cyclotomic trace (see
below). Since for a connective $\s$-algebra $A$ we have that
$K_1(A)\cong K_1(\pi_0A)$, we get little new: $K_i(A)\cong
K_i(\pi_0A)$ for all $i\leq 1$.  Likewise for $TC$.

Topological cyclic homology $TC$ is taken to be integral topological cyclic
homology as defined by Goodwillie \cite{MSRI}, but appears in this paper only
after profinite completion, and so agrees with the product over all primes
$p$ of the $p$-completion of the $p$-typical version $TC(-;p)$ appearing in
\cite{BHM}.  The cyclotomic trace is given as in \cite{dundas04}.

All displayed diagrams commute.

\section{Sufficient conditions on the path components}\label{sec:suff}

In this section we prove proposition \ref{theo:strong} (and a multirelative
version, corollary \ref{cor:strong}) that turns out to be stronger than the
main theorem \ref{thm:MainIntro}.  We do not require the square of
$\s$-algebras to be homotopy cartesian, but rather impose criteria on the
path components.
\begin{prop}
\label{theo:strong} Let \kata\ be a diagram
\begin{equation*}
  \begin{CD}
    A^0 @>>> A^1 \\
    @VVV @VVV\\
    A^2 @>>> A^{12}
  \end{CD}
\end{equation*}
of connective \s-algebras such that $\pi_0A^1 \hp
\pi_0A^{12}$ is surjective and the induced map of  rings
$$\pi_0A^0\to \pi_0A^1\times_{\pi_0A^{12}}\pi_0A^2$$
is a surjection with nilpotent kernel.
  Then the cube
$$K(\kata)\hp TC(\kata)$$
  induced by the trace map is homotopy cartesian after profinite completion.
\end{prop}

\begin{proof}
Let $F$ be the profinite completion of the homotopy fiber of the cyclotomic
trace $K\to TC$.
Since $\pi_0A^1 \hp
\pi_0A^{12}$ is surjective, Geisser and Hesselholt's theorem implies that the
square
$$
\begin{CD}
  F(\pi_0A^1\times_{\pi_0A^{12}}\pi_0A^2)@>>>F(\pi_0A^1)\\
@VVV@VVV\\
F(\pi_0A^2)@>>>F(\pi_0A^{12})
\end{CD}
$$
is homotopy cartesian.
The assumption that $\pi_0A^0\to \pi_0A^1\times_{\pi_0A^{12}}\pi_0A^2$
is a surjection with nilpotent kernel, opens for the use of McCarthy's theorem
\cite{mccarthy97} and we may conclude that
$$
F(\pi_0A^0)\to F(\pi_0A^1\times_{\pi_0A^{12}}\pi_0A^2)
$$
is an equivalence. Hence the square $F(\pi_0\mathcal A)$ is homotopy
cartesian.

Now, by \cite{dundas97:_relat_k}, each of the vertical maps in the cube
$$
\begin{CD}
  F(\mathcal A)\\@VVV\\F(\pi_0\mathcal A)
\end{CD}
$$ are equivalences, and the result follows.
\end{proof}

The above results automatically give theorems about $n$-cubes for $n\geq 1$.
Recall that if $S$ is a finite set, then an {\em $S$-cube} is a functor from
the category $\mathcal PS$ of subsets of $S$, and that if $|S|$ is the
cardinality of $S$, one often uses the term $|S|$-cube.  Hence a $0$-cube is
an object, a $1$-cube is a map and a $2$-cube is a commuting square.

\begin{cor}\label{cor:strong}
  Let $\mathcal A$ be an $S$-cube of connective $\s$-algebras such that for
all $U\subseteq S$ the canonical map
$$p^U\colon\pi_0A^U\to\Lim[U\subsetneq T\subseteq S]\pi_0A^T$$
is surjective, and in addition that $p^{\emptyset}$ has nilpotent kernel.
Then the $(|S|+1)$-cube
$$K(\kata)\hp TC(\kata)$$
induced by the cyclotomic trace is homotopy cartesian after profinite
completion.
\end{cor}
\begin{proof} Note that the surjectivity condition on the cube is
  symmetric in the sense that the condition is satisfied for all
  subcubes. In particular, all maps in the cube are $0$-connected. By
  the same reasoning as in proposition \ref{theo:strong} we may
  immediately reduce to the case of discrete rings. For concreteness,
  let $S=\{1,\dots,n\}$, and assume by induction that the corollary
  has been proven for cubes of cardinality less than $n$.

Let $\mathcal A[\emptyset]$ be the cartesian $(n-1)$-cube obtained by
restricting the functor $\mathcal A$ to $\mathcal P\{1,\dots,n-1\}$ and replacing
$A^\emptyset$ with $\Lim[\emptyset\neq T\subseteq\{1,\dots,n-1\}]A^T$,
and let $\mathcal A[n]$ be the cartesian $(n-1)$-cube obtained by
restricting $\mathcal A$ to the complement of $\mathcal
P\{1,\dots,n-1\}$ and replacing $A^{\{n\}}$ with $\Lim[\{n\}\subsetneq
T\subseteq\{1,\dots,n\}]A^T$.  Then by induction, the corollary applies
to $\mathcal A[\emptyset]$, $\mathcal A[n]$ and to the square
$$
  \begin{CD}
    A^\emptyset@>>> A^{\{n\}}\\@VVV@VVV\\\CDlim[\emptyset\neq
T\subseteq\{1,\dots,n-1\}]A^T@>>>\CDlim[\{n\}\subsetneq
T\subseteq\{1,\dots,n\}]A^T
  \end{CD}
$$
\end{proof}

Notice that the conditions in the corollary are unnecessary restrictive.  If
for instance $n=3$ we see that demanding that \eg $A^{\{1\}}\to A^{\{1,2\}}$,
$A^{\{1,3\}}\to A^{\{1,2,3\}}$, $A^{\{3\}}\to
A^{\{1,3\}}\times_{A^{\{1,2,3\}}}A^{\{2,3\}}$, and
$A^\emptyset\to\Lim[\emptyset\neq T]A^T$ are surjective (and the last map has
a nilpotent kernel) is enough to conclude that $K(\mathcal A)\to TC(\mathcal
A)$ is cartesian after profinite completion.  There are many variants.

\section{Homotopy cartesian squares and $\pi_0$}
\label{sec:reduction}

Theorem \ref{thm:MainIntro} now follows immediately from proposition
\ref{theo:strong} and 
\begin{prop}
   \label{prop:square-zero}Let \kata\ be a homotopy cartesian diagram
   of connective \s-algebras
\begin{equation*}
  \begin{CD}
    A^0 @>g'>> A^1 \\
    @Vf'VV @VVfV\\
    A^2 @>>g> A^{12}
  \end{CD}
\end{equation*}
such that $\pi_0A^1 \hp
\pi_0A^{12}$ is surjective.
   Then the induced  map
   $$h\colon\pi_0A^0\hp\pi_0A^1\times_{\pi_0A^{12}}\pi_0A^2$$
   is a surjection with square zero kernel.
 \end{prop}

\begin{proof}
First we reduce the proof to the corresponding statement for
simplicial rings as found in lemma \ref{lemma:square-zero}.

  Since all $\s$-algebras and the vertical fibers are connective, we
may use $\Gamma$-spaces as our model for spectra, and monoids under
the smash product of Lydakis \cite{Lydakis} as our model for
$\s$-algebras. For details, see \cite{dundas:local-structure}
chapter II.

Let $H$ be the Eilenberg-Mac Lane construction sending a simplicial
ring to a connective $\s$-algebra.  The functor $\tilZ$ which sends a
pointed set $X$ to the free abelian group $\tilZ X=\Z[X]/\Z[*]$
extends to an endofunctor on the category of connective $\s$-algebras.
Furthermore, there is a functor $R$ from connective $\s$-algebras to
simplicial rings and a natural chain of stable equivalences connecting
$\tilZ$ and $HR$. Proves may be found in the published version
\cite[proposition 3.5]{dundas97:_relat_k} or more directly
applicable in \cite[corollary II.2.2.5]{dundas:local-structure}.

Now, the functor $\tilZ$ preserves homotopy cartesian diagrams of
$\Gamma$-spaces, and so if $\mathcal A$ is a homotopy cartesian
diagram of connective $\s$-algebras, then $\tilZ\mathcal A$ is a homotopy
cartesian
diagram of connective $\s$-algebras which is equivalent to $H$ of a
homotopy cartesian diagram $R\mathcal A$ of simplicial rings.

Furthermore, the obvious map $\mathcal A\to\tilZ\mathcal A$ is
$1$-connected (see \eg \cite[proposition 3.3]{dundas97:_relat_k}), and
so we get isomorphisms  
$$
\pi_0\mathcal A\cong
\pi_0\tilZ \mathcal A\cong \pi_0HR\mathcal A\cong \pi_0R\mathcal A
$$
of squares of rings.
\end{proof}

\begin{lemma}
   \label{lemma:square-zero}Let \kata\ be a homotopy cartesian diagram
   of simplicial rings
\begin{equation*}
  \begin{CD}
    A^0 @>g'>> A^1 \\
    @Vf'VV @VVfV\\
    A^2 @>>g> A^{12}
  \end{CD}
\end{equation*}
such that $\pi_0A^1 \hp
\pi_0A^{12}$ is surjective.
   Then the induced  map
   $$h\colon\pi_0A^0\lhp\pi_0A^1\times_{\pi_0A^{12}}\pi_0A^2$$
   is a surjection with square zero kernel.
 \end{lemma}

\begin{proof}
Chasing long exact sequences of homotopy groups yields that $h$ is
surjective.

In proving that the kernel of
$\pi_0A^0\hp\pi_0A^1\times_{\pi_0A^{12}}\pi_0A^2$ is square zero, the
idea is to pick two elements in $\ker(h)\subseteq\pi_0A^0$ and show,
by making an appropriate choice of representatives, that the product
of the representatives is homotopic to 0 in $A^0$. This implies that
the kernel is square zero. The proof is an exercise in manipulating
simplicial homotopies and we refer to \cite{may92:_simpl} for
details. For homotopic simplices $x$ and $y$ in a simplicial abelian
group $G$, we write $x\sim y$. If $x$ and $y$ happen to be
zero-simplices, then being homotopic means that there is a 1-simplex
$z$ with $d_0z=x$ and $d_1z=y$.

We may assume that $A^1\to A^{12}$ is a fibration. According to
\cite[p. II.3.10]{quillen67:_homot}, maps of simplicial
groups are surjective if and only if they are both fibrations and
$0$-connected, and so the assumption that $\pi_0A^1 \hp\pi_0A^{12}$ is
surjective implies that $A^1\to A^{12}$ is a surjection.

Let
$[u_0]\in\ker(h)$
be
represented by $u_0\in A^0_0$.
Then $h([u_0])=([f'u_0],[g'u_0])=0$ in the
pullback, and
$f'(u_0)\sim 0$ and $g'(u_0)\sim 0$ as 0-simplices in $A^2$ and
$A^1$ respectively. The homotopies are given by 1-simplices $u_2\in
A^2_1$ with $d_0u_2=f'u_0$ and $d_1u_2=0$ and $u_1\in A^1_1$ with
$d_0u_1=g'u_0$ and $d_1u_1=0$. These simplices correspond to based
maps $u_2\colon I\hp A^2$ and $u_1\colon I\hp A^1$ respectively
and by abuse of notation we name the maps after their corresponding
simplices. As $u_0$ is a 0-simplex of $A^0$ it corresponds to a
based map $u_0\colon S^0\hp A^0$. All this fits into the following
diagram of pointed simplicial sets
\begin{equation*}
  \label{eq:rep-kernel-naiv}
     \begin{CD}
       I @<<< S^0 @>>> I \\
       @Vu_2VV  @VVu_0V    @VVu_1V \\
       A^2 @<f'<< A^0 @>g'>> A^1
     \end{CD}
\end{equation*}
and it represents an element in $\ker(h)$.

Since $f'$ is a fibration of simplicial rings, we may lift
$u_2\colon I \hp A^2$ to a based map $u\colon I \hp A^0$.  It will
usually not be
compatible with the map $u_0\colon S^0\hp A^0$, but we do have that
$[u_0]=[u_0-d_0u]$ and $f'(u_0-d_0u)=0$, showing that it is enough to
consider the situation where $u_2=0$, that is, diagrams of the form
   \begin{equation}
     \label{eq:rep-kernel-smart}
     \begin{CD}
       * @<<< S^0 @>>> I \\
       @VVV  @VVu_0V    @VVu_1V \\
       A^2 @<f'<< A^0  @>g'>> A^1
     \end{CD}
   \end{equation}
   This diagram induces a (based) map $u_{12}\colon S^1\hp A^{12}$,
   and is our ``appropriate choice''.

   Let $[u_0]$ and $[v_0]$ be any two elements in $\ker(h)$ and pick
   representatives for them as in diagram (\ref{eq:rep-kernel-smart}).

   Consider the map $s_0(g'u_0)\cdot v_1\colon I \hp A^1$. The dot
   denotes multiplication in $A^1$.  The map is a simplicial homotopy
   from $g'(u_0)\cdot g'(v_0)$ to 0 since \[ d_1(s_0(g'u_0)\cdot
   v_1)=d_1s_0(g'u_0)\cdot d_1(v_1) = g'(u_0)\cdot 0 = 0.  \] and \[
   d_0(s_0(g'u_0)\cdot v_1) = d_0s_0(g'u_0)\cdot d_0(v_1) =
   g'(u_0)\cdot g'(v_0).  \] Because we picked a representative
   $u_0\in \ker(f')$ we get \[ f(s_0(g'u_0)\cdot v_1) =
   s_0(fg'(u_0))\cdot f(v_1) = s_0(gf'(u_0))\cdot f(v_1) = 0\cdot
   f(v_1) = 0.  \] The equations above show that we get a well-defined
   and based map \[(0,s_0(g'u_0)\cdot v_1)\colon I \hp
   A^2\times_{A^{12}}A^1,\] determining a simplicial homotopy from
   $(0,g'u_0\cdot g'v_0)$ to 0. Under the weak equivalence $A^0\simeq
   A^2\times_{A^{12}}A^1$, the element $(0,g'u_0\cdot g'v_0)$
   corresponds to the product $u_0\cdot v_0$, showing that it is
   homotopic to 0.
\end{proof}

\begin{ex}
\label{ex:kerh_non_triv}
An example of the situation in lemma \ref{lemma:square-zero} where the
kernel of $h$ is non-trivial may be helpful.
Consider the diagram of simplicial rings
\begin{equation*}
  \begin{CD}
    \Z[\varepsilon]/\varepsilon^2 @>>> \Z \\
    @VVV @VVV\\
    \Z @>>> \Z[S^1]
  \end{CD}
\end{equation*}
in which both maps to $\Z$ are projections and both maps from $\Z$ are
inclusion into the simplicial ring $\Z[S^1]$. All rings in the diagam
but $\Z[S^1]$ are discrete. The diagram is homotopy cartesian which is
easily checked as the zeroth homotopy groups are the only non-trivial
homotopy groups of the fibers. In this case
$$
h\colon \Z[\varepsilon]/\varepsilon^2 \lhp
\pi_0\Z\times_{\pi_0\Z[S^1]}\pi_0\Z \cong \Z 
$$
is the projection with square zero $\ker h =
\Z\langle\varepsilon\rangle$, the infinite cyclic group generated by
$\varepsilon$.
\end{ex}

The connectivity hypothesis on $f^1$ is annoying in that it makes it
difficult to state minimal hypotheses for good multirelative
versions.
As a crude corollary of the main result one has the
following;

\begin{cor}
\label{cor:main}
Let $\mathcal A$ be a homotopy cartesian $S$-cube of connective
$\s$-algebras such that for all $U\subseteq S$ the canonical map
$$
p^U\colon A^U\to\holim[U\subsetneq T\subseteq S]A^T
$$
is $0$-connected.  Then the $(|S|+1)$-cube
$$
K(\kata)\hp TC(\kata)
$$
induced by the cyclotomic trace is homotopy cartesian after profinite
completion.
\end{cor}
Note that $p^\emptyset$ is an equivalence (and thus 0-connected)
because $\mathcal A$ is assumed to be homotopy cartesian. When
$U=S$, the homotopy limit is taken over the empty set and
$p^S\colon A^S\to *$ is clearly 0-connected.

 \begin{proof}[Proof of corollary \ref{cor:main}]
   The proof of this corollary is exactly as the proof of corollary
\ref{cor:strong}, except that you remove $\pi_0$ (and replace the limits by
homotopy limits or replace the cube with a fiber cube so that limits and
homotopy limits agree up to stable equivalence).
 \end{proof}

\begin{remark}\label{rem:3cube}
Corollary \ref{cor:main} is not optimal. For instance if $n=3$, it
would also suffice that the maps $A^\emptyset\to A^{\{3\}}$,
$A^{\{1\}}\to A^{\{1,2\}}$ and $A^{\{2,3\}}\to A^{\{1,2,3\}}$ were
$0$-connected (in addition to homotopy cartesianness of the
cube). Note that this condition is actually not contained in the one
given in the corollary, but is one of the many variants possible. We
spell out this example.

Let $F$ be the profinite completion of the fiber of the cyclotomic trace.  We
may assume all maps are fibrations. Then $F$ applied to the squares
$$
\begin{CD}
  A^{\{1,3\}}\times_{A^{\{1,2,3\}}}A^{\{2,3\}}@>>>A^{\{2,3\}}\\
@VVV@VVV\\
A^{\{1,3\}}@>>>A^{\{1,2,3\}}
\end{CD},\qquad
\begin{CD}
  A^{\{1\}}\times_{A^{\{1,2\}}}A^{\{2\}}@>>>A^{\{2\}}\\
@VVV@VVV\\
A^{\{1\}}@>>>A^{\{1,2\}}
\end{CD}
$$
give homotopy cartesian squares.

Consider the square
$$
\begin{CD}
  A^\emptyset@>>>A^{\{3\}}\\
@VVV@VVV\\
A^{\{1\}}\times_{A^{\{1,2\}}}A^{\{2\}}@>>>A^{\{1,3\}}\times_{A^{\{1,2,3\}}}A^
{\{2,3\}}
\end{CD}.
$$
This square is homotopy cartesian since the entire cube is, and by assumption
the top map is $0$-connected.  Since everything is connective it follows that
the bottom map is $0$-connected too, and so the main theorem applies again to
show that $F$ applied to this square is homotopy cartesian.  Collecting the
pieces we get that $F$ applied to the cube is homotopy cartesian.

Theorem \ref{thm:MainIntro} implies that for $n=2$, we only need
$f^1$ (or $f^2$) to be 0-connected, but the condition in
\ref{cor:main} requires both to be 0-connected, this shows again that
the statement of \ref{cor:main} is not optimal.
\end{remark}

\bibliographystyle{plain}

\end{document}